\documentclass[11pt]{amsart}
\usepackage{amsmath,amssymb,amsthm}
\usepackage[margin=1.15in]{geometry}

\theoremstyle{plain}
\newtheorem{theorem}{Theorem}[section]
\newtheorem{proposition}[theorem]{Proposition}
\newtheorem{lemma}[theorem]{Lemma}
\newtheorem{corollary}[theorem]{Corollary}
\newtheorem{result}[theorem]{Result}
\theoremstyle{definition}
\newtheorem{definition}[theorem]{Definition}
\newtheorem{example}[theorem]{Example}
\theoremstyle{remark}

\newcommand{\Z}{\mathbb{Z}}

\newcommand{\C}{\mathbb{C}}
\newcommand{\F}{\mathbb{F}}
\newcommand{\Ghat}{\widehat{G}}
\newcommand{\Khat}{\widehat{K}}
\newcommand{\Hhat}{\widehat{H}}
\newcommand{\ord}{\operatorname{ord}}
\newcommand{\Tr}{\operatorname{Tr}}

\newcommand{\Aut}{\operatorname{Aut}}
\newcommand{\lcm}{\operatorname{lcm}}
\newcommand{\BH}{\mathrm{BH}}

\begin{document}

\title{Butson Hadamard Matrices from Pairs of Gauss Sums}
\date{19 September 2026}

\author{Cheng Yiin Ong}
\address{School of Physical and Mathematical Sciences, Nanyang Technological
University, Singapore 637371, Republic of Singapore}
\email{ongc0106@e.ntu.edu.sg}

\author{Bernhard Schmidt}
\address{School of Physical and Mathematical Sciences, Nanyang Technological
University, Singapore 637371, Republic of Singapore}
\email{bernhard@ntu.edu.sg}


\begin{abstract}
Let $q$ be a prime power.
We give a construction of Butson Hadamard matrices $\BH((\F_q\times \F_q,+),\lcm(6,d))$
for any divisor $d>1$ of $q-1$.
Most character values of the group ring elements corresponding to these matrices involve
products of two Gauss sums over $\F_q$ (with non-quadratic Gauss sums occurring) and this property distinguishes our result
from all previously known constructions of group invariant Butson Hadamard matrices.
In fact, our $\BH((\F_q\times \F_q,+),\lcm(6,d))$ matrices arise as a special case of a more general construction based on 
what we call ``Butson-Jacobsthal functions'' that resembles the construction
of Hadamard matrices from Jacobsthal matrices. 
\end{abstract}

\maketitle

%%%%%%%%%%%%%%%%%%%%%%%%%%%%%%%%%%%%%%%%%%%%%%%%%%%%%%%%%%

\section{Introduction}

Write $\zeta_h=e^{2\pi i/h}$ and let $\mu_h$ be the group of $h$th roots of unity.
An $n \times n$ matrix $H$ with entries in $\mu_h$ is a \emph{Butson Hadamard matrix}
if it satisfies $HH^* = nI$, where $H^{*}$ denotes the conjugate transpose of $H$ and $I$ is the identity matrix.
We call $H$ a \emph{$\BH(n,h)$ matrix}.
Let $G$ be a multiplicatively written group of order $n$.
Using the elements of $G$ as row and column indices,
we say that a matrix $A = (a_{g,k})_{g,k \in G}$ is \emph{$G$-invariant}
if $a_{gl, kl} = a_{g, k}$ for all $g, k, l \in G$.
A $G$-invariant $\BH(n, h)$ matrix is called a \emph{$\BH(G, h)$ matrix}.

Recently, group invariant Butson Hadamard matrices have been studied intensively;
see \cite{Sch19} for a survey.
The so-called generalized bent functions from $G$ to $\Z_h$ are special types of $\BH(G,h)$ matrices
\cite[Prop.~2.3]{Sch19}, and $\BH(G,h)$ matrices are also connected to relative difference sets in a suitable extension of $G$
\cite[Prop.~2.7]{Sch19}.
Nonexistence
results for generalized bent functions, and thus for $\BH(G,h)$ matrices, were
obtained by Leung and Wang \cite{LW20} and by Leung, Li, and Mao \cite{LLM23}.
The methods used in their work rest on results on vanishing sums of roots of
unity \cite{CJ76, LL00} and on the structure of subsets without unique differences in finite
abelian groups \cite{LS22}.
Hiranandani and Schlenker \cite{HS16} classified the
$\BH(\Z_p,p)$ matrices for primes $p$. Leung, Chue, and Zhao \cite{LCZ} removed
the restriction on the order of the roots of unity and proved that all
$\BH(\Z_p,h)$ matrices are monomially equivalent to Fourier matrices.

For groups of square order, which is the case studied in this paper, the situation is different,
and it is one of our aims to show that $\BH(G,h)$ matrices exist in abundance in this case.
The simplest relevant case is $G=\Z_p\times \Z_p$ where $p$ is a prime.
Under a semiprimitivity assumption, 
 all $\BH(\Z_p\times \Z_p, h)$ matrices can be determined explicitly and arise from a spread construction \cite{DS26}.
Our main result, in particular, implies that, in the absence of semiprimitivity,
there are many further $\BH(\Z_p\times \Z_p, h)$ matrices with character values of a new type.
Moreover, our construction extends to numerous further families of abelian groups of square order.

Notably, our construction
shares several features with the classical construction of Hadamard
matrices from Jacobsthal matrices due to Paley \cite{Pal33}
(and this is our reason for using the term ``Butson-Jacobsthal function'').
Let $q\equiv3\pmod4$ be a prime power and let $\eta$ be the quadratic character
of $\F_q^{\times}$, extended by $\eta(0)=0$. The \emph{Jacobsthal matrix}
\[
  Q=\bigl(\eta(x-y)\bigr)_{x,y\in\F_q}
\]
has vanishing diagonal, entries $\pm1$ off the diagonal, and satisfies
\begin{equation}\label{eq:jacprop}
  QQ^{T}=qI-J,\qquad QJ=JQ=0,
\end{equation}
where $I$ and $J$ are the identity and the all-one matrix of size $q$.
Paley's Hadamard matrix of order $q+1$ is obtained from $Q$ by bordering:
\begin{equation}\label{eq:paley}
  \begin{pmatrix}
      1 & \mathbf 1^{T}\\[2pt]
      -\mathbf 1 & I+Q
    \end{pmatrix}.
\end{equation}

In our construction, the role of the quadratic character $\eta$ is played by 
what we call \emph{Butson-Jacobsthal functions}.
If $f$ is such a function on a group $K$ of order $q$, then the matrix
\begin{equation}\label{eq:bjmatrix}
  Q_f=\bigl(f(xy^{-1})\bigr)_{x,y\in K}
\end{equation}
has vanishing diagonal, entries in $\mu_h$ off the diagonal, and satisfies
\begin{equation}\label{eq:bjprop}
  Q_f^{\phantom{*}}Q_f^{*}=qI-J,\qquad Q_fJ=JQ_f=0 ,
\end{equation}
that is, \eqref{eq:jacprop} with the transpose replaced by the conjugate
transpose.
While the Paley construction borders a single copy of $Q$ by a row and a column of $\pm1$,
our construction takes the Kronecker product of two matrices $Q_{f_1}, Q_{f_2}$
obtained from Butson-Jacobsthal functions $f_1,f_2$,
and borders it in both directions at once.
In matrix form, Theorem \ref{thm:main} shows that
\begin{equation}\label{eq:matrixform}
 \zeta_6\,(J\otimes I)
   +\zeta_6^{5}\,(I\otimes J)
   +\gamma\,(Q_{f_1}\otimes Q_{f_2})
\end{equation}
is a $\BH(K\times H,h)$ matrix whenever $6\mid h$, $\gamma\in\mu_h$, and
$f_1,f_2$ are Butson-Jacobsthal functions on $K,H$, respectively.

%%%%%%%%%%%%%%%%%%%%%%%%%%%%%%%%%%%%%%%%%%%%%%%%%%%%%%%%

\section{Preliminaries}
For a finite abelian group $G$, let $\widehat{G}$ denote its group of complex characters.
The following is a standard result \cite[Ch.~VI, Lem.~3.5]{be}. 
\begin{result}[Inversion Formula] \label{fourier}
Let $G$ be a finite abelian group and $X=\sum_{g\in G} a_gg$ with $a_g\in \C$. Then
$$
a_g = \frac{1}{|G|} \sum_{\chi\in \widehat{G}} \chi(Xg^{-1}) \text{ for all } g\in G.
$$
\end{result}

Throughout, we use the group ring formulation of \cite{Sch19} and pass between
matrices and group ring elements without further comment; see also \cite{HU06}. 
For a ring $R$, the elements of $R[G]$ have the form $D=\sum_{g\in G}a_gg$, $a_g\in R$,
and the $a_g$'s are the \emph{coefficients} of $D$.
For $R=\Z[\zeta_h]$, define
\[
  \Phi(D)=\bigl(a_{gk^{-1}}\bigr)_{g,k\in G}.
\]
Then $\Phi$ is an  $R$-algebra isomorphism  from $R[G]$ onto the algebra of
$G$-invariant matrices with entries in $R$. Its inverse sends a $G$-invariant
matrix $A$ to $\sum_{g\in G}a_{g,1}\,g$. Moreover $\Phi(D^{(-1)})=\Phi(D)^{*}$,
where $D^{(-1)}=\sum_{g\in G}\overline{a_g}\,g^{-1}$.
Consequently, $\Phi(D)$ is a
$\BH(G,h)$ matrix if and only if all $a_g$ are $h$th roots of unity and
$DD^{(-1)}=|G|$ in $R[G]$.
In this case, we say that $D$ corresponds to a $\BH(G,h)$ matrix.

\begin{lemma}[\cite{Sch19}, Lemma 2.1]\label{lem:crit}
Let $a_g\in\mu_h$ for all $g\in G$ and put $D=\sum_{g\in G}a_g g$.
Then $D$ corresponds to a $\BH(G,h)$ matrix if and only if
$DD^{(-1)}=|G|$, which holds if and only if
\begin{equation}\label{eq:crit}
|\chi(D)|^{2}=|G| \quad\text{for all }\chi\in\Ghat.
\end{equation}
\end{lemma}

We define two $\BH(G,h)$ matrices with corresponding group ring elements $D,E$
as \emph{equivalent} if
\begin{equation}\label{eq:equiv}
E= \pm \zeta_h^i\,g\,\alpha(D)
\end{equation}
for some integer $i$, $g\in G$, and $\alpha\in \Aut(G)$.
Note that we do \emph{not} use monomial equivalence of matrices here
(since monomial transformations do not preserve group invariance).

%%%%%%%%%%%%%%%%%%%%%%%%%%%%%%%%%%%%%%%%%%%%%%%%%%%%%%

\section{The Construction}

Let $q$ be a prime power and let  $\varphi$ be a nontrivial multiplicative character of $\F_{q}^{\times}$, extended by $\varphi(0)=0$. 
Then $\sum_{x\in\F_{q}}\varphi(x)=0$ and 
$$
\biggl|\sum_{x\in\F_{q}}\varphi(x)\alpha(x)\biggr|^{2}=q
$$
for every nontrivial character $\alpha$ of $(\F_q,+)$, 
a well-known property of Gauss sums \cite[Thm.~5.11]{LN97}.
Our construction works with the following generalization of multiplicative characters of finite fields.

\begin{definition}\label{def:gl}
Let $K$ be a finite abelian group. We call a function
$f:K\to\mu_h\cup\{0\}$ a \emph{Butson-Jacobsthal function} if $f(1)=0$,
\[
\sum_{a\in K}f(a)=0,
\qquad\text{and}\qquad
\Bigl|\sum_{a\in K}f(a)\chi(a)\Bigr|^{2}=|K|
\]
for every nontrivial character $\chi$ of $K$.
\end{definition}

\begin{theorem}\label{thm:main}
Let $G=K\times H$, where $K$ and $H$ are finite abelian groups with $|K|=|H|$,
and let $h$ be a positive integer divisible by $6$. Let $f_1:K \to \mu_h\cup\{0\}$ and $f_2: H\to \mu_h\cup\{0\}$ be
Butson-Jacobsthal functions and let $\gamma\in\mu_h$. 
Then any $E$ equivalent to  
\begin{equation}\label{eq:main}
D=1+\zeta_6\sum_{x\neq1}(x,1)+\zeta_6^{5}\sum_{y\neq1}(1,y)
+\gamma\sum_{x\neq1,\,y\neq1}f_1(x)f_2(y)\,(x,y)
\end{equation}
corresponds to a $\BH(G,h)$ matrix.
\end{theorem}

\begin{proof}
Write $v=|K|$.
Note that all coefficients of $D$ are $h$th roots of unity. 
Thus, by Lemma \ref{lem:crit}, it suffices to show $|\chi(D)|^{2}=v^{2}$ for all $\chi\in\Ghat$. 
Every $\chi\in\Ghat$ has the form $\chi(x,y)=\alpha(x)\beta(y)$ with $\alpha\in\Khat$ and $\beta\in\Hhat$. 
Write
\[
S(\alpha)=\sum_{x\in K}f_1(x)\alpha(x),\qquad
T(\beta)=\sum_{y\in H}f_2(y)\beta(y) .
\]
As  $f_1:K \to \mu_h\cup\{0\}$ and $f_2: H\to \mu_h\cup\{0\}$ are Butson-Jacobsthal functions,
we have the following: $S(\alpha)=0$ if $\alpha$ is trivial and $|S(\alpha)|^2=v$ if $\alpha$ is nontrivial;
similarly, $T(\beta)=0$ if $\beta$ is trivial and $|T(\beta)|^2=v$ if $\beta$ is nontrivial.
Since $f_1(1)=f_2(1)=0$, we have
\begin{equation}\label{eq:four}
\chi(D)=1+\zeta_6\sum_{x\neq1}\alpha(x)+\zeta_6^{5}\sum_{y\neq1}\beta(y)
+\gamma S(\alpha)T(\beta).
\end{equation}
Recall that $\sum_{x\neq1}\alpha(x)$ equals $v-1$ for trivial $\alpha$ and
$-1$ otherwise. We use $\zeta_6+\zeta_6^{5}=1$ throughout.
Assume first that $\alpha$ and $\beta$ are trivial. Then
$S(\alpha)=T(\beta)=0$ and \eqref{eq:four} gives $\chi(D)=1+(v-1)(\zeta_6+\zeta_6^{5})=v$.
Assume next that $\alpha$ is nontrivial and $\beta$ is trivial. Then
$T(\beta)=0$ and \eqref{eq:four} gives
\[
\chi(D)=1-\zeta_6+(v-1)\zeta_6^{5}
=\bigl(1-\zeta_6-\zeta_6^{5}\bigr)+v\zeta_6^{5}=v\zeta_6^{5}.
\]
In the same way, $\chi(D)=v\zeta_6$ if $\alpha$ is trivial and $\beta$ is
nontrivial.
Finally, assume that both $\alpha$ and $\beta$ are nontrivial. 
Then the first
three terms of \eqref{eq:four} contribute $1-\zeta_6-\zeta_6^{5}=0$, and thus
\begin{equation}\label{eq:val}
\chi(D)=\gamma\,S(\alpha)T(\beta).
\end{equation}
Hence $|\chi(D)|^{2}=|S(\alpha)|^{2}|T(\beta)|^{2}=v^{2}$.
In all four cases we have $|\chi(D)|^{2}=v^{2}$. 
\end{proof}

%%%%%%%%%%%%%%%%%%%%%%%%%%%%%%%%%%%%%%%%%%%%%%%%%%%%%
%%%%%%%%%%%%%%%%%%%%%%%%%%%%%%%%%%%%%%%%%%%%%%%%%%%%%%%

\section{Butson-Jacobsthal functions}\label{sec:gl}

We first rephrase the definition of Butson-Jacobsthal functions in terms of group rings.
Note that we write $K$ for the group ring element $\sum_{x\in K}x$.

\begin{lemma}\label{lem:gr}
Let $K$ be a finite abelian group, $f:K\to\mu_h\cup\{0\}$, and $X=\sum_{x\in K}f(x)x$. 
Then $f$ is a Butson-Jacobsthal function if and only if $f(1)=0$ and
\begin{equation}\label{eq:gr}
XX^{(-1)}=|K|-K .
\end{equation}
\end{lemma}

\begin{proof}
As all nonzero values of $f$ are roots of unity, we have
$\chi(X^{(-1)})=\overline{\chi(X)}$ and thus
$\chi(XX^{(-1)})=|\chi(X)|^{2}$ for all $\chi\in\Khat$. Furthermore,
$\chi(|K|-K)=0$ if $\chi$ is trivial and $\chi(|K|-K)=|K|$ if $\chi$ is nontrivial.    
Now the equivalence of  \eqref{eq:gr} and
 the two conditions in Definition \ref{def:gl} follows from Result \ref{fourier}.
\end{proof}

Note that \eqref{eq:gr} forces $|\{x:f(x)\neq0\}|=|K|-1$ by comparison of the
coefficients of the identity. Thus a Butson-Jacobsthal function vanishes at $1$ and
nowhere else.

\begin{example}\label{ex:char}
Let $K=(\F_{q},+)$ with $q=p^{n}$ and let $\varphi$ be a nontrivial character
of $\F_{q}^{\times}$, extended by $\varphi(0)=0$. Recall the Gauss sum
$G(\varphi)=\sum_{x\in\F_{q}^{\times}}\varphi(x)\zeta_p^{\Tr(x)}$, where $\Tr$
is the absolute trace of $\F_q$. We have $\sum_{x\in\F_q}\varphi(x)=0$ by
orthogonality, and
$\sum_{x\in\F_{q}}\varphi(x)\zeta_p^{\Tr(ux)}
=\overline{\varphi}(u)G(\varphi)$ for $u\neq0$. Since
$|G(\varphi)|^{2}=q$  \cite[Thm.~5.11]{LN97}, 
we conclude that $\varphi$ is a Butson-Jacobsthal function with values in $\mu_{\ord(\varphi)}\cup\{0\}$.
\end{example}

The next result provides Butson-Jacobsthal functions with values in $\{0,\pm1\}$. 
Recall that a subset $D$ of an abelian group $K$ of order $v$
 is a \emph{skew Hadamard difference set} if $|D|=\frac{v-1}{2}$, $DD^{(-1)}=\frac{v+1}{4}+\frac{v-3}{4}K$, and
$D\cup D^{(-1)}=K\setminus\{1\}$ is a disjoint union, and that $D$ is a
\emph{partial difference set with Paley parameters} if $|D|=\frac{v-1}{2}$, $D=D^{(-1)}$, and
$DD^{(-1)}=\frac{v-1}{4}(1+K)-D$.

\begin{proposition}\label{prop:pm1}
Let $K$ be a finite abelian group.
If $D$ is  a skew Hadamard difference set or a partial difference set with Paley parameters in $K$,
then $X=K-1-2D$ corresponds to a Butson-Jacobsthal function $f:K\to\{0,\pm 1\}$ via Lemma
\ref{lem:gr}.
\end{proposition}

\begin{proof}
We first note that the coefficients of $X$ are all in $\{0,\pm 1\}$ and the coefficient of $1$
in $X$ is $0$, since $D$ has coefficients $0,1$ and $1\not\in D$.
Thus, by Lemma \ref{lem:gr}, $X$ corresponds  to a Butson-Jacobsthal function $f:K\to\{0,\pm 1\}$
if and only if $XX^{(-1)}=|K|-K$. 
It is straightforward to check that this follows from the equations that $D$ satisfies.
\end{proof}

Various known constructions  are available as input for Proposition \ref{prop:pm1}.
Numerous families of skew Hadamard difference sets exist, see
\cite{Pal33,be,DY06,DWX07,WQWX07,Muz10,FX12,DPW13,Mom13}.
Paley type partial difference sets exist in abelian groups
which are not elementary abelian and even in groups which are not $p$-groups,
see Polhill \cite{Pol10}. Note that the groups $K$ and $H$ in Theorem
\ref{thm:main}  need not be isomorphic, and that no field structure is
required.

Combining Theorem \ref{thm:main} with Example \ref{ex:char}, we obtain the
following explicit form over finite fields.

\begin{corollary}\label{cor:field}
Let $q=p^{n}$ be a prime power, let $\varphi_1,\varphi_2$ be nontrivial
characters of $\F_{q}^{\times}$, and let $h$ be divisible by
$\lcm(6,\ord(\varphi_1),\ord(\varphi_2))$. Let $\gamma\in\mu_h$ and put
$G=(\F_q\times\F_q,+)\cong(\Z_p)^{2n}$. Then
\[
D=1+\zeta_6\sum_{x\neq0}(x,0)+\zeta_6^{5}\sum_{y\neq0}(0,y)
+\gamma\sum_{x,y\neq0}\varphi_1(x)\varphi_2(y)\,(x,y)
\]
corresponds to a $\BH(G,h)$ matrix. Writing
$\chi_{(u,v)}(x,y)=\zeta_p^{\Tr(ux+vy)}$, we have
\begin{equation}\label{eq:fieldval}
\chi_{(u,v)}(D)=\gamma\,\overline{\varphi_1}(u)\,\overline{\varphi_2}(v)\,
G(\varphi_1)G(\varphi_2)
\qquad\text{for } u,v\neq0 .
\end{equation}
\end{corollary}

\begin{proof}
Apply Theorem \ref{thm:main} with $f_1=\varphi_1$ and $f_2=\varphi_2$. Formula
\eqref{eq:fieldval} follows from \eqref{eq:val} and
$S(\chi_u)=\overline{\varphi_1}(u)G(\varphi_1)$, see Example \ref{ex:char}.
\end{proof}

%%%%%%%%%%%%%%%%%%%%%%%%%%%%%%%%%%%%%%%%%%%%%%%%%%%%%%%%%

\section{Comparison with previously known constructions}\label{sec:new}

To our knowledge, all previously constructed $\BH(G,h)$ matrices with $G$ an abelian group of square order,
with the sole exception of \cite{Sch26}, 
have the property that all character values of the corresponding group ring element are divisible by $\sqrt{|G|}$. 
In other words, 
if $D\in \Z[\zeta_h][G]$ corresponds to  a  previously constructed $\BH(G,h)$  matrix (different from those in \cite{Sch26})
with abelian $G$ of order $u^2$, then
\begin{equation} \label{naive}
\chi(D) \equiv 0\ (\bmod\ u) \text{ for all } \chi \in \widehat{G}.
\end{equation}
In particular, \eqref{naive} holds for all $\BH(G,h)$ matrices constructed in 
\cite{bac, Duc21, DS19, HS16, kum, ro1, ro2, SWX21}. 
In this section, we show that \eqref{naive} does \emph{not} hold for  many of the matrices constructed in Theorem \ref{thm:main}.
Thus, these matrices  are new. 

We call a prime $p$ \emph{semiprimitive} \emph{modulo $h$}
if $(p,h)=1$ and $p^i \equiv -1\ (\bmod\ h)$ for some integer $i$.
Recall that a Gauss sum $G(\varphi)$ over $\F_q$ is called \emph{pure} if some
positive power of $G(\varphi)$ is rational. Since $|G(\varphi)|^2=q$, this holds
if and only if $G(\varphi)/\sqrt q$ is a root of unity. The quadratic
character $\eta$ always has a pure Gauss sum, as $G(\eta)^2=\eta(-1)q$. 
More generally, $G(\varphi)$ is pure whenever $p$ is semiprimitive modulo $\ord(\varphi)$, see \cite{BMW82}. 
Conversely,  if $\ord(\varphi)$ is a prime power and $p$ is not semiprimitive modulo $\ord(\varphi)$,
then $G(\varphi)$ is not pure \cite{Cho62,Eva81}.

\begin{lemma}\label{lem:pure}
Let $q$, $\varphi_1$, $\varphi_2$, $h$, $\gamma$, and $D$ be as in Corollary
\ref{cor:field}, and suppose that $G(\varphi_1)$ is pure and $G(\varphi_2)$ is
not pure. Then
\[
\chi_{(u,v)}(D)\not\equiv0\pmod q\qquad\text{for all }u,v\in\F_q^{\times}.
\]
\end{lemma}

\begin{proof}
Let $u,v\ne0$ and suppose that $\chi_{(u,v)}(D)\equiv 0\pmod q$.
Write $\chi_{(u,v)}(D) = q Y$ with $Y\in \Z[\zeta_{ph}]$. 
Note that $|Y|=1$ since $|\chi_{(u,v)}(D)|^2=q^2$. 
Thus, $Y$ is a root of unity by Kronecker's theorem \cite{kronecker}. 
On the other hand, 
\[
\chi_{(u,v)}(D) =\gamma\,\overline{\varphi_1}(u)\,\overline{\varphi_2}(v)\,G(\varphi_1)G(\varphi_2)
\]
by \eqref{eq:fieldval} and $G(\varphi_1)/\sqrt{q}$ is a root of unity since $G(\varphi_1)$ is pure. 
Hence
$$
\frac{G(\varphi_2)}{\sqrt{q}} 
=\overline{\gamma}\,\varphi_1(u)\,\varphi_2(v)  \frac{\chi_{(u,v)}(D)}{G(\varphi_1)\sqrt{q}}
= \overline{\gamma}\,\varphi_1(u)\,\varphi_2(v)Y  \frac{\sqrt{q}}{G(\varphi_1)}
$$
is a root of unity, contradicting the non-purity of $G(\varphi_2)$.
\end{proof}

\begin{corollary}\label{cor:pure}
Let $p$ be a prime and let $m\ge3$ be a divisor of $p-1$. Then there is a
$\BH(\Z_p\times\Z_p,\lcm(6,m))$ matrix with corresponding group ring element $D$ 
such that $\chi(D)\not\equiv0\pmod p$
for all $\chi\in\Ghat$ which are nontrivial on both factors.
\end{corollary}

\begin{proof}
Apply Lemma \ref{lem:pure} with the quadratic character for $\varphi_1$ and $\varphi_2$ of order
$m$. Note that $G(\varphi_2)$ is not pure by \cite{Cho62}. 
\end{proof}

It is interesting to compare Corollary \ref{cor:pure} with the following result
from \cite{DS26,SWX21}:

\begin{result} \label{spread}
Let $p$ be an odd prime, let $h$ be a positive integer,
let $G = \Z_p\times \Z_p$, let $U_0, \dots, U_p$ be the subgroups of $G$ of order $p$,
and let $\eta_0, \dots, \eta_p\in \mu_h$ with $\sum_{i=0}^{p} \eta_i = 1$. Then
\begin{equation} \label{eq:spread}
	D = \sum_{i=0}^{p} \eta_i U_i
\end{equation}
corresponds to a $\BH(G, h)$ matrix.
Moreover, if $p$ is semiprimitive modulo $h$, then every $\BH(G, h)$ matrix is equivalent to one
of the form (\ref{eq:spread}). 
\end{result}

A natural question that arises from Result \ref{spread} is
whether all $\BH(\Z_p\times \Z_p, h)$ matrices with $(p,h)=1$ must be of the form \eqref{eq:spread},
even when $p$ is not semiprimitive modulo $h$.
Note that \eqref{eq:spread} implies that \eqref{naive} holds for all matrices that 
are equivalent to a matrix constructed in Result \ref{spread}.
Thus, Corollary \ref{cor:pure} provides many examples of $\BH(\Z_p\times \Z_p, h)$ matrices
that are not equivalent to any matrix of the form \eqref{eq:spread}.

%%%%%%%%%%%%%%%%%%%%%%%%%%%%%%%%%%%%%%%%%%%%%%%%%%%%%
\section*{Acknowledgement}

Part of this work was carried out with the assistance of Anthropic's Claude Fable 5
and DeepSeek V4. 
The systems were used to run exploratory computations on $\BH(\Z_5\times \Z_5,12)$ 
matrices and to search for generalizations. 
The constructions and proofs presented here were written 
by the authors, who take full responsibility for the content.

%%%%%%%%%%%%%%%%%%%%%%%%%%%%%%%%%%%%%%%%%%%%%%%%%%%%%

\end{document}